\documentclass[12pt]{article}

\usepackage{amsmath,amssymb,amsfonts,latexsym,amsthm}
\usepackage{amsmath,amssymb,amsxtra,epsf,amscd,graphics,color,array,ulem,etex}
\usepackage[all]{xy}
\usepackage[all]{xy}
\usepackage{mathrsfs}
\usepackage[dvips]{geometry}
\usepackage{array}
\usepackage{epsf}
\usepackage{epsfig}
\usepackage{verbatim,ifthen, fancyvrb,xkeyval,xcolor,epic,eepic,rotating}

\usepackage{mathrsfs}
\usepackage[dvips]{geometry}
\usepackage{array}
\usepackage{epsf}
\usepackage{epsfig}
\newtheorem{thm}{Theorem}
\theoremstyle{definition}
\newtheorem{defi}{Definition}
\theoremstyle{plain}

\theoremstyle{definition}

\theoremstyle{plain}
\newtheorem{lemma}[thm]{Lemma}
\newtheorem{prop}[thm]{Proposition}
\newtheorem{cor}[thm]{Corollary}
\theoremstyle{remark}
\newtheorem{rem}{Remark}
\theoremstyle{plain}

\theoremstyle{remark}

\newtheorem{rems}[thm]{Remarks}

\newcommand{\iso}{\overset{\sim}{\rightarrow}}

\newcommand{\ad}{\operatorname{ad}}
\newcommand{\nc}{\newcommand}
\nc{\Spec}{\operatorname{Spec}}
\nc{\Index}{\operatorname{Index}}
\nc{\GKdim}{\operatorname{GKdim}}
\nc{\Stab}{\operatorname {Stab}}
\nc{\Fract}{\operatorname {Fract}}
\nc{\End}{\operatorname {End}}
\def\q{\mathfrak q}
\def\h{\mathfrak h}
\def\g{\mathfrak g}
\def\a{\mathfrak a}

\def\s{\mathfrak s}
\def\l{\mathfrak l}
\def\b{\mathfrak b}

\title{ADAPTED PAIRS AND WEIERSTRASS SECTIONS
\footnote{Work supported in part by Israel Science Foundation Grant No.  710724.}}

\author{Florence Fauquant-Millet 
\footnote{email : florence.millet@univ-st-etienne.fr}
and Anthony Joseph
\footnote{email : anthony.joseph@weizmann.ac.il}}

\date{March 2015}

\begin{document}

\maketitle

\begin{abstract}

Adapted pairs and Weierstrass sections are central to the invariant theory associated to the action of an algebraic Lie algebra $\mathfrak a$ on a finite dimensional vector space $X$. In this $\mathfrak a$ need not be a semisimple Lie algebra.  Here their general properties are described particularly when $\mathfrak a$ is the canonical truncation of a biparabolic subalgebra of a simple Lie algebra and $X=\mathfrak a^*$.

\end{abstract}


\



\section{Introduction.}\label{1}

Throughout this paper the base field $\bf k$ is assumed algebraically closed of characteristic zero and
$\a$ denotes an algebraic Lie algebra over $\bf k$ with $\bf A$ its adjoint group. Let $S(\mathfrak a)$ denote the symmetric algebra of $\mathfrak a$ and set $Y(\mathfrak a)=S(\mathfrak a)^{\bf A}$ the algebra of polynomials on $\a^*$ invariant under the action of $\bf A$.

\subsection{}\label{1.1}


Suppose for the moment that $\mathfrak a$ is semisimple. Then there are some remarkable classical theorems describing $Y(\mathfrak a)$.  First the theorem of Chevalley \cite {C} which asserts that $Y(\mathfrak a)$ is polynomial.  Secondly the theorem of Kostant \cite {Ko} which provides a ``linearisation" of the generators of $Y(\mathfrak a)$.
The first result depended on the triangular decomposition of $\mathfrak a$ and further made a crucial use of the Weyl group $W$.  The second used the existence of a principal $\mathfrak {sl}_2$ triple.  These structures are special to semisimple Lie algebras and so for a long time it was not suspected that similar results could hold for other families of Lie algebras.

\

The above situation changed dramatically following \cite {FJ0} which provided a huge and varied family of algebraic Lie algebras $\mathfrak a$ for which $Y(\mathfrak a)$ is polynomial.  Indeed it was shown that $Y(\mathfrak a)$ is polynomial for almost any ``truncated" parabolic subalgebra of a simple Lie algebra $\mathfrak g$ in particular for all truncated parabolics when $\mathfrak g$ is of type $A$ or type $C$.  Moreover the invariant generators turn out to be extremely complicated expressions even for $\mathfrak g$ of type $A$.  They are known explicitly in only a few cases. Shortly after this result was extended to biparabolic subalgebras (which are the intersection of two parabolic Lie subalgebras of $\g$ whose sum is $\g$) in \cite {J0}, \cite {J0.1}.  By convention we exclude all biparabolics which are reductive (for example $\g$ itself).

\

Following this work it was shown (\cite{J1}) for all truncated biparabolics in type $A$ that one can linearize the invariant generators in the sense of Kostant even though an appropriate $\mathfrak {sl}_2$ triple fails to exist.

\

 Partly inspired by our work, Panyushev, Premet and Yakimova \cite {PPY} showed that $Y(\mathfrak a)$ is polynomial for almost any centraliser $\a=\mathfrak g^x: x \ \text {nilpotent}$ in a simple Lie algebra $\mathfrak g$ in particular for all cases when $\mathfrak g$ of type $A$ or type $C$.  From our present point of view this case is not so interesting as $Y(\mathfrak g^x)$ is less mysterious and indeed is obtained (in the good cases) as the last non-vanishing partial derivatives of elements of  $Y(\mathfrak g)$ with respect to the second nilpotent element of the $\mathfrak {sl}_2$ triple containing $x$.  Moreover linearisation of $Y(\mathfrak g^x)$ in type $A$ follows rather closely that for $Y(\mathfrak g)$ itself \cite {J4}.

\

Generally truncated biparabolics and centralizers are quite different.  Again our methods are quite different to those of \cite {PPY} even though both work perfectly in types $A$ and $C$ but  fail in general.  However one might remark that there is one case when a centralizer $\mathfrak g^x$ is a truncated parabolic, namely when $x$ is a highest weight vector (for the adjoint representation).  Then $Y(\mathfrak g^x)$ can be shown with some extra efforts \cite {PPY}, \cite {J4} to be polynomial outside type $E_8$, whilst for type $E_8$, it fails to be polynomial \cite {Y}.

\subsection{}\label{1.2}

Previous to our work Popov \cite {P} had considered the problem of linearisation of invariant generators in the case when $\mathfrak a$ acts reductively on a finite dimensional vector space $X$.
Here one may assume without loss of generality that $\mathfrak a $ is semisimple.

Define a linear subvariety of $X$ to be a subset of the form $y+V$, with $y \in  X$ and $V$ a subspace of $X$.

Following Popov \cite [2.2.1] {P} define a Weierstrass section for the action of $\mathfrak a$ on $X$ to be a linear subvariety $y+V$ of $X$ such that the restriction of $S(X^*)$ to $y+V$ induces an isomorphism of $Y(X^*):=S(X^*)^{\bf A}$ onto the algebra $R[y+V]$ of regular functions on $y+V$. Of course the existence of a Weierstrass section implies the  polynomiality of $Y(X^*)$, which is itself a rather rare phenomenon.

 Let $X_{reg}$ denote the set of regular elements of $X$, that is to say all $y \in X$ with stabilizer $\mathfrak a^y$ in $\mathfrak a$ of minimal dimension called the index $\ell_X(\a)$ of $\a$ relative to $X$.  One may remark that $X_{reg}$ is open dense in $X$.
 
Assume that  $\mathfrak a$ is a simple Lie algebra and that $X$ is a simple $\mathfrak a$ module.  By a case by case verification Popov \cite [2.2.10]{P} showed that there exists $y \in X_{reg}$ and a semisimple endomorphism $h \in \mathfrak a$ of $X$ such that $h.y =-y$.  However this may fail if $X$ is not a simple module \cite [2.2.16, Example 3]{P}. In the case that it holds a Weierstrass section obtains as $y+V$ where $V$ is an $h$-stable complement to $\a.y$ in $X$.

 Some further example of Weierstrass sections occur in the work of Reeder, Levy,Yu and Gross \cite {RLYG} coming from the Dynkin gradation $\oplus_{i \in \mathbb Z} \mathfrak g_i$ of a semisimple Lie algebra $\mathfrak g$.  Here $\mathfrak a=\mathfrak g_0$ which is of course reductive and $X=\mathfrak g_1$.

\subsection{}\label{1.3}

The notion of a Weierstrass section immediately extends to any finite dimensional Lie algebra $\mathfrak a$, though we shall assume that $\mathfrak a$ is algebraic as this is more relevant to questions concerning invariants. Moreover it is natural to assume in this context that $S(X^*)$ admits no proper semi-invariants, which implies that $(\Fract S(X^*))^{\bf A} = \Fract S(X^*)^{\bf A}$.  Thus under these assumptions it follows from a theorem of Chevalley-Rosenlicht (\cite {R} and see \cite [Lemme 7]{Dix0}) that the Gelfand-Kirillov dimension (or growth rate) of the invariant algebra $S(X^*)^{\bf A}$ is just $\ell_X(\mathfrak a)$.

We denote $\ell_{\mathfrak a^*}(\mathfrak a)$ by $\ell(\mathfrak a)$, called the index of $\mathfrak a$.

Again we define an adapted pair for the action of $\a$ on $X$ to be a pair $(h,\,y) \in \mathfrak a \times X_{reg}$ satisfying $h.y=-y$. Since $\a$ is algebraic, we can and do assume that $h$ is a semisimple endomorphism of $X$.

Given an adapted pair $(h,\,y)$ for the action of $\a$ on $X$, then under the hypotheses of the first paragraph we have a natural candidate for Weierstrass section following \ref {1.2}.  Indeed let $V$ be an $h$-stable complement to $\mathfrak a.y$ in $X$.  Then $\dim V=\ell_X(\mathfrak a)$ and the restriction of $Y(X^*)$ to $y+V$ is an injection $\varphi$ of $Y(X^*)$ into $R[y+V]$.  A fortiori the restriction $\Phi$ of $Y(X^*)$ to $\textbf{k}y+V$ is injective with image contained in $R[\textbf{k}y+V]^h$.  On the other hand it is not at all obvious if $\varphi$ is also surjective.  Notice that $\Phi$ is a map of graded vector spaces.

If $\mathfrak a$ is reductive, then following \cite {J0.01} one may define a Letzter map $\mathscr L$ of $R[\textbf{k}y+V]^h$ into $Y(X^*)$ which is again a map of graded vector spaces. Suppose $\mathscr L$ is injective.  Then $\Phi\mathscr L \in \End R[\textbf{k}y+V]^h$ is injective on each graded subspace and since these are finite dimensional it must also be surjective. Thus $\Phi$ and hence $\varphi$ is surjective.    Seemingly to prove the required injectivity is not easy and indeed in the co-adjoint case (that is when $X=\a^*$) it would give an alternative and quite different proof of the Chevalley and Kostant theorems mentioned in \ref {1.1}.

On the other hand if $\a$ is not reductive, the Letzter map cannot be constructed and we have much less reason to believe that $\varphi$ is surjective even in the co-adjoint case as we further discuss in the next section.

\subsection{}\label{1.4}

From now on we assume that $X=\mathfrak a^*$, that is to say we just consider co-adjoint action. Recall the index of $\mathfrak a$ (denoted $\ell(\mathfrak a)$) is defined to be the $\dim \mathfrak a^\eta:\eta \in \mathfrak a^*_{reg}$.  Given $x \in \a$, we let $\ad x$ denote the map $y\mapsto [x,y], \forall y \in \a$, that is to say the adjoint action of $\a$ on itself , as well as its transpose that is to say the co-adjoint action of $\a$ on $\a^*$. An adapted pair, resp. a Weierstrass section, for the action of $\a$ on $\a^*$ will be called simply an adapted pair, resp. a Weierstrass section, for $\a$.

We further assume that $S(\mathfrak a)$ has no proper semi-invariants, that is to say that $Y(\a)$ is equal to the Poisson semicentre $Sy(\a)$ of $S(\a)$, which is the vector space generated by the semi-invariants of $S(\a)$ under the action of $\a$.  This is the case when $\a$ is the canonical truncation (see \cite[Def. 7.1, Sec. B, Chap. I]{F}) of an algebraic Lie algebra $\b$ and moreover one has $ Sy(\b)=Y(\a)=Sy(\a)$. This result is referred to as a generalization of a lemma of Borho in \cite{RV}. In addition one has
 $$ \dim \mathfrak b + \ell(\mathfrak b) = \dim \mathfrak a +  \ell (\mathfrak a).$$

 The above is reviewed in \cite {OV} and \cite [Prop. 9.7, Sec. B, Chap. I]{F}.

Assume that $Y(\mathfrak a)$ is polynomial.  Then extending an observation in \cite {FJ2}, Ooms and Van den Bergh \cite {OV} established a sum rule on the degrees of the homogeneous generators of $Y(\mathfrak a)$.  In \cite {JS} we gave a simpler and slightly more general result based on ideas of Panyushev and furthermore used this technology to establish the surjectivity of $\varphi$ defined in \ref{1.3} for an adapted pair $(h,\,y)$ for $\a$ and $V$ defined as in \ref{1.3}.   On the other hand dropping the assumption that $Y(\mathfrak a)$ is polynomial, we found that in a slightly wider context surjectivity can fail. Our hope had been to construct an adapted pair to explain Yakimova's counter-example described in \ref {1.1}.

Further use of this technology will be made here.

Call an algebraic Lie algebra $\a$ \textit{regular} if $S(\a)$ admits no proper semi-invariants and if $Y(\a)$ is polynomial (necessarily on $\ell(\a)$ generators).

\subsection{}\label{1.5}

Following (\cite[7.3]{J2}) a slice $\mathscr S$ to the action of ${\bf A}$ on $\a^*$ is
a locally closed subvariety of $\a^*$ such that the following two conditions are satisfied :\par

$(i)$ ${\bf A}. \mathscr S$ is dense in $\a^*$ \par
$(ii)$ every orbit in ${\bf A}. \mathscr S$ cuts $\mathscr S$ at exactly one point and then transversally. 

 (The latter means that, for all $s\in\mathscr S$, denoting by $T_{s,\,\mathscr S}$ the tangent space in $\mathscr S$ at $s$, the sum  $\a. s + T_{s,\,\mathscr S}$ is direct, see \cite[7.1]{J2}).

Following \cite[7.4]{J2} a slice $\mathscr S$ to the action of ${\bf A}$ on $\a^*$ is called affine if its closure
$\overline{\mathscr S}$ is a linear subvariety of $\a^*$, that is to say of the form $y+V$ as described in \ref {1.2}.

A linear subvariety $y+V$ of $\a^*$ will be called an affine slice for $\a$ if $y+V$ itself (and not just a locally closed subset $\mathscr S$ of $\a^*$ such that $\overline{\mathscr S}=y+V$) satisfies conditions $(i)$ and $(ii)$ above.

Note that, if the locally closed subset $\mathscr S$ of $\a^*$ is a slice with closure $\overline{\mathscr S}=y+V$ then $y+V$ satisfies also condition $(i)$ but not necessarily condition $(ii)$ above (see \ref{7.11.7}).

Assume in the rest of this subsection that $S(\a)$ admits no proper semi-invariants.

Lemma \ref{7.11.2} shows that a Weierstrass section $y+V$ for $\a$ is an affine slice for $\a$.  Notably this does not need $y \in \mathfrak a^*$ to be regular.

The question of the converse is much more delicate. Following \cite [7.6]{J2} we say that $y+V$ is a rational slice to the action of ${\bf A}$ on $\a^*$ if the restriction of $Y(\a)$ to $y+V$ is injective and induces an isomorphism of rings of fractions.  By \cite [7.7]{J2} $y+V$ gives rise to an affine slice if and only if it is a rational slice. Furthermore if $\mathscr S$ is a slice to the action of ${\bf A}$ on $\a^*$  with closure $y+V$, then the restriction of functions is an injection from $Y(\a)$ into $R[y+V]$. Moreover if we identify $Y(\mathfrak a)$ with its image in $R[y+V]$, then (through \cite [7.4]{J2})  there exists $d\in Y(\a)$ such that $Y(\a)[d^{-1}]=R[y+V][d^{-1}]$.  (We shall recap some of this theory here since it was rather poorly presented in \cite [Sect. 7]{J2}).

On the other hand even when $y$ is regular and hence part of an adapted pair, we do not expect (see \ref {1.4}), under these equivalent conditions on $y+V$, that $Y(\mathfrak a)$ is polynomial.  Again even under the additional assumption that $Y(\a)$ is polynomial we cannot say that an affine slice for $\a$ is a Weierstrass section for $\a$.  Indeed we are only able to assert this under some further assumption(s) (\ref {7.11.3} - \ref {7.11.5}) the simplest of which is that $(y+V)\setminus (y+V)_{reg}$ has codimension $\geq 2$.

\subsection{}\label{1.6}

Since we are just considering the co-adjoint action of $\a$, it is convenient to replace $y$ above by the corresponding Greek letter $\eta$.

Let $(h,\,\eta)$ be an adapted pair for $\a$.  Again  we shall denote by $V$ an $h$-stable complement to $\a.\eta$ in $\a^*$. One checks that duality induces a non-degenerate pairing of $V$ with $\a^\eta$. The eigenvalues of $h$ on $V$ are called the exponents of $\a$ relative to $\eta$.  Then, since $\eta$ is regular, the number of exponents (counted with multiplicities) is just $\ell(\a)$.  On the other hand it is not obvious if the exponents themselves are independent of $\eta$.

 Now suppose that $\a$ is regular.  Then we have the remarkable fact \cite [Cor. 2.3]{JS} that exponents are the degrees of the homogeneous generators minus one and in particular independent of the choice of $\eta$.  This gives a way of showing that the existence of an adapted pair does not imply regularity.  It was discussed in \cite {JS}.


\subsection {}\label{1.7}

The construction of an adapted pair is an elementary question of linear algebra, but an extremely difficult one.  In \ref {7.8} we give it a more geometric interpretation and speak of equivalence classes of adapted pairs for truncated biparabolics.

Let $\a$ be a truncated biparabolic subalgebra of a simple Lie algebra $\g$.  If $\g$ is of type $A$ or type $C$, then $\a$ is regular.  If $\g$ is of type $A$ then $\a$ admits an adapted pair $(h,\,\eta)$ \cite {J1}.  This was a difficult result and $\eta$ can only be described in terms of meanders. Moreover rather many equivalence classes were constructed and so far there is no way to decide if all were found.  In \cite {FJ3} it was shown for any adapted pair $(h,\eta)$ constructed in \cite {J1} the second element $\eta$ is always an image of a \textit{regular nilpotent} element of $\mathfrak g^*$.  This was a very difficult result. Significantly it allows one to bring the Weyl group into the description of these adapted pairs.  Moreover it allows one to rationalize the existence of a Weierstrass section $\eta+V$ when $\eta$ is not regular \cite {J2}.

A truncated biparabolic in type $C$ need not admit an adapted pair, not even a Weierstrass section \cite {J2}, for example the truncated Borel subalgebra.

By contrast one may conjecture that a biparabolic subalgebra always admits an adapted pair; but this is less interesting.  Rather trivially a Frobenius Lie algebra always admits an adapted pair.  A biparabolic subalgebra may be Frobenius and already in type $A$ their number grows rather rapidly. A Borel subalgebra is Frobenius if and only if $-1$  belongs to the Weyl group, for example in type $C$.
%
\subsection{}\label{1.8}
 Let $Y(\mathfrak a)_+$ be the span of homogeneous vectors of $Y(\a)$ of degree $>0$.
Let $\mathscr N(\a)\subset\a^*$ be
 the zero set  of the ideal $S(\mathfrak a)Y(\mathfrak a)_+$ of $S(\mathfrak a)$ .  Since $S(\mathfrak a)Y(\mathfrak a)_+$ has homogeneous generators, it follows that
$\mathscr N(\a)$ is stable under the multiplicative group $\bf {k}^*$ of the field $\bf k$. Since $\bf {k}^*$ is a connected algebraic group, each irreducible component of
$\mathscr N(\a)$ is $\bf{ k}^*$ stable.

 Suppose $\eta \in \mathfrak a^*$ satisfies $(\ad h)(\eta)=\lambda \eta$, for some
 $h \in \mathfrak a,\, \lambda \in \bf {k}^*$.  Take $p \in Y(\mathfrak a)_+$ homogeneous of positive degree $d$.  Then $0=((\ad h)(p))(\eta)=-d\lambda p(\eta)$,
which forces $p(\eta)=0$, that is
$\eta\in\mathscr N(\a)$.  In particular if $(h,\eta)$ is an adapted pair for $\a$, then $\eta \in \mathscr N(\a)_{reg}:=\mathscr N(\a)\cap\a^*_{reg}$.  On the other hand there is no guarantee that $\mathscr N(\a)_{reg}$ is non-empty.

It is clear that $\bf{A}$ acts by simultaneous conjugation on the set of adapted pairs for $\a$.  It also acts on $\mathscr N(\a)$.
In \ref{defequivAP} we give a definition of equivalent adapted pairs for $\a$ involving only the second terms of the adapted pairs.
When $\a$ is a regular truncated biparabolic, we show (\ref{equivpara}) that equivalent adapted pairs for $\a$ in the above sense are also equivalent in the sense of simultaneous conjugation by $\bf A$ and a simple criterion is given to determine when two adapted pairs for $\a$ are equivalent, namely Corollary \ref {7.9}.

  In Proposition \ref {7.6} we show, when $\a$ is regular, that the map $(h,\eta) \mapsto \overline{\bf{A}.\eta}$ is a bijection of equivalence classes of adapted pairs for $\a$ (in the first sense) onto the irreducible components of the variety $\mathscr N(\a)$ admitting a dense orbit.  
  Hence, when $\a$ is moreover a truncated biparabolic, this bijection involves the equivalence classes of adapted pairs in the second sense.

\subsection{}\label{1.9}

Let $(h,\eta)$ be an adapted pair for $\a$.  An intriguing question is whether the eigenvalues of $\ad h$ on $\mathfrak a$ are all integer-valued.  In \ref {7.10} we note that this can be false when $\mathfrak a$ is Frobenius.  However when $\a$ is regular there are no known counter-examples.  A main difference in this latter case is that the eigenvalues of $h$ on $\a^\eta$ are (negative) integers.  In \ref {7.10} we use this result to show that if $\mathfrak a$ is a truncated biparabolic subalgebra in a simple Lie algebra $\mathfrak g$ of type $A$, then indeed this question has a positive answer.



\section{Adapted pairs}\label{2}

 Recall the definition of an adapted pair $(h,\eta)$ for $\a$ given in \ref{1.3}, using the notation of \ref {1.6}.

\subsection{Equivalent adapted pairs.}

\subsubsection{Definition}\label{defequivAP}

Let $(h,\,\eta)$ be an adapted pair for $\a$. For all $a\in{\bf A}$, the pair $(a.h,\,a.\eta)$ is also an adapted pair for $\a$.

\begin{defi}
Let $(h,\,\eta)$ and $(h',\,\eta')$ be two adapted pairs for $\a$. We say that these adapted pairs are equivalent if there exists an element
$a\in{\bf A}$ such that $a.\eta=\eta'$.

\end{defi}

\begin{rem}
It would be more appropriate to say that two adapted pairs $(h,\,\eta)$ and $(h',\,\eta')$ are equivalent if and only if they belong to the same orbit according to the diagonal action or simultaneous conjugation by ${\bf A}$, that is to say if and only if there exists $a\in{\bf A}$ such that
$\eta'=a.\eta$ and $h'=a.h$.
Actually we will show this in~\ref{equivpara} when $\a$ is the canonical truncation $\q_{\Lambda}$ of a biparabolic subalgebra of $\g$ simple. But for this, we need first to give a presentation of $\eta$ (see  \ref{7.1} $(*)$ below) and of the vector space $V$ defined in \ref {1.6} (see \ref{2.2.2} $(**)$ below), when $(h,\,\eta)$ is an adapted pair for $\q_{\Lambda}$.

\end{rem}

\subsubsection{}\label{7.6}

Now we will give the bijection announced in \ref{1.8}.

  Let $A$ be a finitely generated commutative algebra. Suppose $a \in A$ is 
 a zero divisor (resp. non invertible and non-divisor of zero). Then by Krull's theorem $\GKdim A/aA$ equals $ \GKdim A$ (resp. $\GKdim A -1$).  Thus if $I$ is an ideal of $A$ 
generated by $s$ algebraically independent generators, the irreducible components of its zero variety all have codimension at 
most $s$. For example taking $A={\bf k}[x,\,y]$, the zero set of the ideal generated by 
$x^2,\,xy$ has a component of dimension $1$ and a component of dimension $0$.

Assume that $\a$ is regular.

  It follows that any irreducible component $C$ of 
 $\mathscr N(\a)$ has dimension at least $d(\mathfrak a):=\dim \mathfrak a - \GKdim {Y}(\mathfrak a)$. For the opposite inequality we need to know that the tangent space $T_{c,\,C}$ at some $c \in C$ has dimension at most $d(\mathfrak a)$.

Let $\{x_i\}_{i=1}^{\dim \mathfrak a}$ be a basis for $\mathfrak a$. Given $f \in S(\mathfrak a)$ its differential $df(c)$ at $c \in \mathfrak a^*$ may be identified with the element $\sum_{i=1}^{\dim \mathfrak a} (\partial f/\partial x_i)(c)x_i$ of $\mathfrak a$.

Let $C$ be an irreducible component of 
$\mathscr N(\a)$.  Let $f_1,\,f_2,\,\ldots,\,f_s$ be a set of generators of the ideal of definition of $C$.  Then the tangent space $T_{c,\,C}$ at $c$ in $C$  may be identified with the orthogonal in $\mathfrak a^*$ of the subspace generated by $df_i(c)\,:\,i=1,\,2,\,\ldots,\,s$.  Thus for the codimension of $C$ to be exactly $\ell(\mathfrak a)$, it suffices to know that for some $c \in C$, the dimension of the space generated by the $df_i(c)\,:\,i=1,\,2,\,\ldots,\,s$ is at least $\ell(\mathfrak a)$.  In this it is enough to take the $f_i\,:\,i=1,\,2,\,\ldots,\,\ell(\a)$ in ${Y}(\mathfrak a)_+$ and to show that the $df_i(c)\,:\,i=1,\,2,\,\ldots,\,\ell(\a)$ are linearly independent. This will be used in the first part of the proof of the Proposition below.



For each subset $Y$ of $\mathfrak a^*$ we let $\overline{Y}$ denote its Zariski closure.

\begin {prop} Let $\mathfrak a$ be an algebraic Lie algebra with ${Sy}(\mathfrak a)={Y}(\mathfrak a)$ polynomial.  The map $(h,\,\eta)\mapsto \overline {\bf{A}.\eta}$ is a bijection of the equivalence classes of adapted pairs for $\a$ onto the irreducible components of 
$\mathscr N(\a)$ of codimension $\ell(\mathfrak a)$ admitting a regular element.
\end {prop}

\begin {proof}
Let $(h,\,\eta)$ be an adapted pair.  By \ref{1.8}, $\eta\in\mathscr N(\a)$. Let $V$ be an $\ad h$ stable complement to 
$(\ad\mathfrak a)(\eta)$ in $\mathfrak a^*$.

 By hypothesis $Y(\mathfrak a)$ is polynomial, on (homogeneous) generators $f_1,\,f_2,\,\ldots,\, f_{\ell(\mathfrak a)}\in{Y}(\mathfrak a)_+$. Then as noted in \cite [Cor. 2.3]{JS} the linear subvariety $\eta +V$ is a Weierstrass section (in the sense of \ref{1.2}) for $\a$.  In the proof one shows (cf \cite [Thm. 6.3]{JS}) that the $df_i(\eta)$ are linearly independent.  Consequently, by what was noticed before the Proposition, the irreducible component of 
$\mathscr N(\a)$ containing $\eta$ has codimension exactly $\ell(\mathfrak a)$.

(This conclusion was obtained in \cite [Prop. 8.5(i)]{J0.7} by a different and more complicated argument).

We conclude  that $\overline{\textbf{A}.\eta}$ is an irreducible component of 
$\mathscr N(\a)$ of codimension $\ell(\mathfrak a)$ containing a regular element, namely $\eta$.

Conversely suppose that $C$ is an irreducible component of 
$\mathscr N(\a)$ of codimension $\ell(\mathfrak a)$ containing some regular element $\eta$. Then $\textbf{A}.\eta$ is open dense in $C$.

Consider the action of $\bf{ k}^*$ on $C$. Take $\lambda \in \bf {k}^*$. Clearly $\lambda \eta$ is also regular in $\mathfrak a^*$.  Then $\textbf{A}.(\lambda \eta)$ is open dense in $C$ and so must coincide with $\textbf{A}.\eta$.

 It follows that the tangent space 
$T_{\eta,\,C}$ contains $\bf {k}\eta$.  Yet this tangent space is just $(\ad \mathfrak a)(\eta)$.  We conclude that there exists $h \in \mathfrak a$ such that $(\ad h)(\eta)=-\eta$. In other words $(h,\,\eta)$ is an adapted pair for $\mathfrak a$.

\end {proof}

\begin{rems}

\

\begin{enumerate}

\item Notice that the bijection of the above Proposition may be obtained from \cite[Prop. 8.5]{J0.7} but only for the particular case of a regular truncated parabolic subalgebra. Here the Proposition is true for any regular  algebraic Lie algebra $\a$.

\item  The conclusion of the above Proposition is that classifying adapted pairs (up to the  equivalence defined in \ref{defequivAP}) is the same as classifying irreducible components of 
$\mathscr N(\a)$ of codimension $\ell(\mathfrak a)$ admitting a regular element
or is the same as classifying irreducible components of $\mathscr N(\a)$ admitting a dense orbit. This does not make adapted pairs easier to find but gives the search for them more meaning.

\item  In the above we have not excluded the possibility of there being irreducible components of 
$\mathscr N(\a)$ of codimension $< \ell(\mathfrak a)$.  More generally we do not know that 
$\mathscr N(\a)$  is equidimensional.

\item It can happen that $\mathscr N(\mathfrak a)$ admits no regular elements.  For example take  $\a$ to be the truncated Borel subalgebra in $\mathfrak {sp}_4(\bf k)$.  Then  $\mathscr N(\a)$ admits no regular elements.  Again take $\a$ to be a truncated biparabolic in $\s\l_n(\bf k)$. Then it can happen \cite [11.3.4]{J1} that 
$\mathscr N(\a)$
  has an irreducible component admitting no regular elements.

\item  Continue to assume that $\a$ is regular.  If 
$\mathscr N(\a)$ is irreducible and $\a$ admits an adapted pair $(h,\,\eta)$, then the above Proposition implies that $\mathscr N(\a)=\overline{{\bf A}.\eta}$.  Then, as in \cite[Proof of Corollary 8.6] {J0.7}, we obtain that $\mathscr N(\a)$ is a complete intersection by an argument following Kostant.  Otherwise a component of $\mathscr N(\a)$ containing the second element of an adapted pair need not be a complete intersection \cite [Sect. 11] {J1}.
\end{enumerate}
\end{rems}

 \subsection{The particular case of a truncated biparabolic subalgebra.}\label{7.1}

 Let $\g$ be a simple Lie algebra, $\h$ a Cartan subalgebra of $\g$, $\Delta$ the set of roots for the pair $(\g,\h)$ and $\pi$ a choice of simple roots.  Fix subsets $\pi^-,\pi^+$ of $\pi$ and let $\mathfrak q_{\pi^-,\pi^+}$ (or simply, $\mathfrak q$) denote the subspace of $\g$ spanned by $\h$ and the root subspaces of $\g$ with roots in $R^-:=(\mathbb N \pi^-\cup -\mathbb N\pi^+)\cap \Delta$. It is called a (standard) biparabolic subalgebra of $\g$.

 In the above framework our standing hypothesis is that
$$\pi^+\cup \pi^-=\pi,\, \pi^+\cap \pi^- \subsetneq \pi.$$

   This excludes $\q$ being a biparabolic subalgebra of a proper semisimple subalgebra of $\g$ and being reductive.  It implies that the invariant subalgebra $Y(\mathfrak q)$ is reduced to scalars \cite [Lemma 7.9]{J0} so is of lesser interest.  However the semi-invariant algebra $Sy(\mathfrak q)$ can be quite large and the latter is often polynomial (see \cite {FJ0} for the parabolic case and \cite {J0.1} in general).  Since $\mathfrak q$ is algebraic, it admits a canonical truncation $\mathfrak q_{\Lambda}$.  This is obtained by simply replacing its Cartan subalgebra $\mathfrak h$ by the common kernel $\h_\Lambda$ in $\h$ of the set $\Lambda$ of characters of $Sy(\q)$.


Consistent with our previous conventions we denote by $\bf Q_{\Lambda}$ the  adjoint group of $\q_{\Lambda}$.

 Unless otherwise specified we shall assume that $\q_{\Lambda}$ is a regular Lie algebra, which reduces to assuming that $Y(\mathfrak q_{\Lambda})$ is polynomial.  Let $d_i\,:\,i=1,\,2,\,\ldots,\,s=\ell(\q_{\Lambda})$ be the degrees of homogeneous generators of $Y(\mathfrak q_{\Lambda})$ and set $e_i=d_i-1$.
 Let $\kappa$ be a Chevalley antiautomorphism for $\mathfrak g$.   Through the Killing form on the latter we may identify $\kappa(\mathfrak q)$ with $\mathfrak q^*$ as an $\mathfrak h$ module.  Then $R:=-R^-$ is the set of (non-zero) roots of $\mathfrak q^*$ identified with $\kappa(\mathfrak q)$. For each $\alpha\in\Delta$, let $x_{\alpha}$ denote the unique (up to a non-zero scalar) non-zero vector of $\g$ of weight $\alpha$.



 Assume that $\mathfrak q_\Lambda$ admits $(h,\,\eta)$ as an adapted pair (with $h$ ad-semisimple).
  Through simultaneous conjugation by $\textbf{Q}_{\Lambda}$ we may assume that $h \in \mathfrak h_\Lambda$.

 Let $x_0$ denote an element of $\mathfrak h_\Lambda$.
 Then we can write $\eta$ in the form
 $\eta = \sum_{\alpha \in S\cup\{ 0\}}x_\alpha$
 for some subset $S$ of $R$.  The condition $(\ad h)(\eta)=-\eta$ means that  $h(\alpha)=-1$, for all $\alpha \in S$ and then that $x_0=0$. 
 Thus
 $$\eta = \sum_{\alpha \in S}x_\alpha \eqno{(*)}.$$

 Let $V$ be an $\ad h$ stable complement to $(\ad \mathfrak q_\Lambda)(\eta)$ in $\mathfrak q_\Lambda^*$.   Recall (cf \cite [2.1]{JS}) that duality restricts to a non-degenerate $\ad h$ equivariant pairing $V\times \mathfrak q_\Lambda^\eta$. Conversely any $\ad h$ stable subspace of $\mathfrak q_\Lambda^*$ non-degenerately paired to $\q_\Lambda^\eta$ is an  $\ad h$ stable complement to $(\ad \mathfrak q_\Lambda)(\eta)$.

 By \ref {1.6} $\{e_i\,:\,i=1,\,2,\,\ldots,\,s\}$ is the set of eigenvalues of $\ad h$ on $V$. Thus the eigenvalues of $\ad h$ on $\mathfrak q_\Lambda^\eta$ are $\leq 0$.

\subsubsection{}\label{2.2.1}

  \begin {lemma}

 \

 (i)  $\mathfrak q_\Lambda^\eta \subset \oplus_{\alpha \in - R}{\bf k}x_\alpha$.

 \

 (ii) One may choose $V$ such that $V \subset \oplus_{\alpha \in  R}{\bf k}x_\alpha$.

 \

 (iii)  $S_{\mid_{\h_\Lambda}}$ spans $\mathfrak h^*_\Lambda$.
 \end {lemma}

 \begin {proof}
   Consider first the zero $\ad h$ eigenspace of $\mathfrak q_\Lambda^\eta$.  It corresponds to the subspace spanned by the degree one invariants, that is to say the centre $\mathfrak z_{\Lambda}$ of $\mathfrak q_\Lambda$.  The latter is stable under $\ad \mathfrak h$ and so has a basis formed from root vectors.  Yet $\mathfrak h_\Lambda \cap \mathfrak z_{\Lambda}=\{0\}$ since $\pi^+\cup\pi^- =\pi$ so $\mathfrak z_{\Lambda} \subset \oplus_{\alpha \in - R}\textbf{k}x_\alpha$.

   Consider an $\ad h$ eigenspace of $\mathfrak q_\Lambda^\eta$ having a strictly negative eigenvalue.  This must be a sum of root vectors with roots in $-R$.  Combined with the previous observation this proves $(i)$.

 $(ii)$ follows from $(i)$ and the remarks following $(*)$ above.

 If $(iii)$ fails, there exists $h' \in \mathfrak h_\Lambda$, $h'\neq 0$, such that $(\ad h')(\eta)= 0$, thus $h' \in \mathfrak q_\Lambda^\eta$. This contradicts $(i)$.

 \end {proof}

 \subsubsection{}\label{2.2.2}
 \begin{rems}
 \

 \begin{enumerate}
\item  Take $\eta$ as in $(*)$ above.  By $(i)$ of the previous Lemma, $h \in \mathfrak h_\Lambda$ is uniquely determined by the condition that $(\ad h)(\eta)=-\eta$.

\item In \cite [Sect. 6]{J1}  adapted pairs $(h,\,\eta)$ for a canonical truncation $\q_{\Lambda}$ of a biparabolic subalgebra in $\g$ simple of type $A$ were constructed and in this construction the element $\eta$ is given by $(*)$ of \ref{7.1} with  the supplementary property that $S_{\mid_{\h_\Lambda}}$ is a \textit{basis} of $\mathfrak h^*_\Lambda$.  This is stronger than $(iii)$ of Lemma \ref{2.2.1}.

When $\eta$ can be presented as above it will be called a minimal presentation and $(h,\,\eta)$ will be called a minimal representative of the equivalence class of adapted pairs for $\q_{\Lambda}$ defined by the adapted pair $(h,\,\eta)$
and the simultaneous conjugation by ${\bf Q}_{\Lambda}$.


\item A slightly better result than Lemma \ref {2.2.1}$(ii)$ was obtained in the construction of adapted pairs for truncated biparabolics in type $A$.  Namely it was shown $V$ (which is \textit{not} unique) can be chosen to have a basis amongst the root vectors of $\mathfrak q^*$, that is to say there exists a subset $T$ of $R$ such that

$$V=\oplus_{\alpha \in T}{\bf k} x_\alpha, \eqno{(**)}$$
is an $\ad h$ stable complement to $(\ad \mathfrak q_{\Lambda})(\eta)$ in $\mathfrak q_{\Lambda}^*$.

\end{enumerate}
\end{rems}

\subsubsection{}\label{7.2}

We shall show that $(**)$ of \ref {2.2.2} can always be assumed to hold in the present more general situation through the following elementary result from linear algebra.

Let $V$ be a vector space of dimension $m < \infty$.  Let $\{v_i\}_{i=1}^m$ be a basis of $V$. Choose a finite index subset $\sqcup_{i=1}^m T_i$ (of the positive integers) and for all $i=1,\,2,\,\ldots,\, m$ vectors $v_{i,\,r} \in V\,: \,r \in T_i$ such that
$$v_i = \sum_{r \in T_i} v_{i,\,r}.$$

\begin {lemma}  For all $i=1,\,2,\,\ldots,\, m$ there exists $r_i \in T_i$ such that $\{v_{i,\,r_i}\}_{i=1}^m$ is a basis for $V$.

\end {lemma}

\begin {proof} By induction on $m$.  It is trivial if $m=1$.  Let $X$ (resp. $Y$) denote the set of all $\{v_{i,r_i}: r_i \in T_i\}_{i=2}^m$ such that $\sum_{i=2}^m {\bf k} v_{i,\,r_i}$ has dimension $m-1$ (resp $<(m-1)$).  By the induction hypothesis $X$ is non-empty.

Fix $\{v_{i,\,r_i}\} \in X$.  If the conclusion of the lemma were false (for $m$ itself) we would have $v_{1,\,r} \in \sum_{i=2}^m {\bf k} v_{i,\,r_i}$, for all $r \in T_1$. Consequently $v_1 \in \sum_{i=2}^m {\bf k} v_{i,\,r_i}$.

The above assertion means that the images of the $v_{i,\,r_i}\,:i\,=2,\,\ldots,\,m$, generate a subspace of dimension $\leq (m-2)$ in $V/{\bf k}v_1$.  This also holds (trivially) if $\{v_{i,\,r_i}\} \in Y$.  It contradicts the induction hypothesis applied to the $m-1$ dimensional vector space $V/{\bf k}v_1$ and proves the lemma.
\end {proof}

\subsubsection{}\label{7.3}

From the above we can obtain our required refinement with the hypotheses and notations of \ref{7.1}.  Let $\{v_i\}$ be a choice of $\ad h$ eigenvectors of $V$.  We can assume that $v_i$ has eigenvalue $e_i$.  By Lemma \ref {2.2.1}$(ii)$ we may write
$$ v_i = \sum_{\alpha \in T_i}c_{i,\,\alpha} x_\alpha,$$
for some $T_i \subset R$ with $h(\alpha)=e_i, \forall \alpha \in T_i$ (and $c_{i,\,\alpha}\in\bf k$ for all $\alpha\in T_i$).

Now apply Lemma \ref {7.2} to $V$ identified with the quotient space $\mathfrak q_{\Lambda}^*/(\ad \mathfrak q_{\Lambda})(\eta)$.  Its conclusion means that we can find a subset $T$ of $R$ such that the image of $V^\prime:=\sum_{\alpha \in T}{\bf k} x_\alpha$ is $\mathfrak q_{\Lambda}^*/(\ad \mathfrak q_{\Lambda})(\eta)$.  Since $V^\prime$ is $h$ stable, we may use it to replace $V$ as a complement to $(\ad \mathfrak q_{\Lambda})(\eta)$ in $\mathfrak q_{\Lambda}^*$.

One may remark that by our construction the eigenvalues of $\ad h$ on $V^\prime$ are still the exponents of $\mathfrak q_{\Lambda}$ (in the sense of \ref{1.6}).  However this is automatic from \cite [Cor. 2.3(i)]{JS}.

\subsubsection{}\label{7.4}

We now give a result which generalizes what was found in the index one case \cite {J0.5} and which is of interest in its own right. Retain the hypotheses and notations of \ref{7.1}.
Assume that $S\subset R$ is defined by $\eta$ as in \ref {7.1}$(*)$ and that
 $V=\sum_{\alpha \in T}{\bf k} x_\alpha :T\subset R$, an $\ad h$-stable complement to $(\ad \mathfrak q_{\Lambda})(\eta)$ in $\mathfrak q_{\Lambda}^*$.
\begin {prop} $S\cup T$ spans $\mathfrak h^*$.
\end {prop}

\begin {proof} Suppose $\eta = 0$, that is $S=\emptyset$.  Since $\eta$ is regular,
 $\mathfrak q_{\Lambda}=\q^\eta_\Lambda$ is commutative and so spanned by $|\pi|$ linearly independent root vectors with roots in $\pi \cup -\pi$.  In this case $T =R$ and this spans $\mathfrak h^*$, hence the assertion.

 From now on assume $\eta \neq 0$.

Recall that $\g$ may be identified with its dual $\g^*$ through the Killing form $K$ on $\g$.
View ${\bf k}\eta +V$ as a subspace of $\g$ and set $V^*=\sum_{\alpha \in T}{\bf k}x_{-\alpha}$ and $y=\sum_{\beta \in S}x_{-\beta}$.
One may remark that $(\ad h)(y)=y$.
Now the sum ${\bf k}\eta +V$ is direct since $\ad h$ has eigenvalue $-1$ on $\eta$ and non-negative eigenvalues on $V$.  
By a similar reason, the sum ${\bf k}y+V^*$ is direct. Moreover  ${\bf k}\eta\oplus V$ and ${\bf k}y\oplus V^*$ are non-degenerately paired through $K$. Identifying ${\bf k}y\oplus V^*$ with a subspace of $\g^*$ through $K$, we may then complete ${\bf k}y\oplus V^*$ to $\g^*$ by the space of functions in $\g^*$ orthogonal to ${\bf k}\eta\oplus V$.


Hence the restriction to ${\bf k}\eta\oplus V$ of a polynomial in ${S}(\g^*)$  is a polynomial on $y$ and the $x_{-\alpha}$, $\alpha\in T$.



Recall that the restriction map induces an algebra isomorphism of ${Y}(\mathfrak q_{\Lambda})$ onto the space of regular functions on $\eta +V$ (\cite [Cor. 2.3]{JS}), which we may identify with ${S}(V^*)$.  For all $\alpha \in T$, let $p_\alpha$ denote the invariant polynomial whose image is the coordinate function defined by $x_\alpha$.  This translates through the above to the formula ${p_\alpha}_{\mid_{\eta+V}} = x_{-\alpha}$. On the other hand since restriction is $h$ equivariant,  ${p_\alpha}_{\mid_{{\bf k}\eta+V}}$ is an $h$-weight vector of zero weight.  Thus
$${p_\alpha}_{\mid_{{\bf k}\eta+V}}=y^{h(\alpha)}x_{-\alpha}+\sum_{I\in\mathscr F} q_Ix^I \eqno{(*)}$$

with $\mathscr F$ a finite set, and for all $I\in\mathscr F$, $q_I\in{\bf k}[y]$ such that $q_I(\eta)=0$ and $x^I$ a monomial in a basis of $V^*$.

Let ${\bf k}\Lambda$ be the $\bf k$ linear span of the $\mathfrak h$-weights of ${Y}(\mathfrak q_{\Lambda})$.  By
definition $\mathfrak h_{\Lambda}= \{h \in \mathfrak h\mid h({\bf k}\Lambda)=0\}$ and conversely ${\bf k}\Lambda=\{\lambda \in \mathfrak h^*\mid\lambda (\mathfrak h_{\Lambda})=0\}$.

Hence in the right hand side of $(*)$ there is an $\h$-weight vector of weight equal to a sum of $-\alpha$ and of ${h(\alpha)}$ elements of $-S$, whereas the left hand side is the restriction to ${\bf k}\eta+V$ of an $\h$-weight vector  of weight belonging to ${\bf k}\Lambda$.  It follows that modulo ${\bf k}S$ the set $T$ spans ${\bf k}\Lambda$. In other words ${\bf k}T+{\bf k}S={\bf k}\Lambda+{\bf k}S$.  On the other hand by Lemma \ref {2.2.1}$(iii)$ one has ${\bf k}S_{\mid_{\mathfrak h_{\Lambda}}}=\mathfrak h^*_{\Lambda}$, whilst ${\bf k}\Lambda$ is the orthogonal of $\mathfrak h_{\Lambda}$ in $\mathfrak h^*$.  Hence ${\bf k}\Lambda+{\bf k}S=\mathfrak h^*$, which by our previous identity gives the assertion of the proposition.

\end {proof}

\subsection{Equivalent adapted pairs for truncated biparabolics.}\label{7.8}

\subsubsection{}\label{equivpara}
Continue with the notations and hypotheses of \ref{7.1}, in particular
 that ${Y}(\mathfrak q_{\Lambda})$ is polynomial and that $\mathfrak q_{\Lambda}$ admits an adapted pair $(h,\,\eta)$.

 Since $h$ is an ad-semisimple element, then up to the diagonal action of
$\bf Q_{\Lambda}$ on $(h,\,\eta)$ we can assume that $h \in \mathfrak h_{\Lambda}$. 

To simplify notation set $\a=\q_\Lambda^{\eta}=\{x\in\q_{\Lambda}\mid(ad\, x)(\eta)=0\},\, \mathfrak z =\mathfrak z_\Lambda$.
Let $\a_i$ be the $\ad h$ eigensubspace of $\a$ of eigenvalue $i$, which we recall is zero unless $i \in -\mathbb N$.
Recall also that $\a_0=\mathfrak z$.
Set $m\in\mathbb N^*$ such that $\a=\displaystyle\oplus_{-m\le i\le 0}\a_i$ and
$\mathfrak u:= \displaystyle\oplus_{-m\le i\le -1}\a_i$. Then the elements of $\mathfrak u$ act by nilpotent derivations on $\q_\Lambda$.

 Let $\textbf{U}$ be the unique irreducible unipotent algebraic subgroup of $\bf Q_{\Lambda}$ with Lie algebra $\ad_{\q_{\Lambda}}(\mathfrak u)$. Observe that, for all $u\in {\bf U}$, $h-u.h\in\mathfrak u$.

\begin {prop}  Let $(h^\prime,\,\eta^\prime)$ be an adapted pair for $\mathfrak q_{\Lambda}$ with
$\eta^\prime \in {\bf Q}_{\Lambda}.\eta$. Then there exists
$q'\in\bf Q_{\Lambda}$ such that
$q'.h^\prime = h,\, q'.\eta^\prime =\eta$.
\end {prop}

\begin {proof}   Since for all $q\in{\bf Q}_{\Lambda}$, $(q.h',\,q.\eta')$ is also an adapted pair (see \ref{defequivAP}), we can assume that $\eta^\prime =\eta$.  Then, adopting the above notations, $h-h^\prime \in \mathfrak a$.

As in \cite [8.10]{J0.7} (following an argument of Kostant, see \cite[3.6]{Ko0}) for every $1\le j\le m$, we can find $w_j\in\mathfrak u$ such that $h'-(exp(\ad\,w_j))(h)\in\a_0+\sum_{i=-m}^{-(j+1)}\a_i$.

Hence we can find $u\in \textbf{U}$ such that $x:=h^\prime-u.h \in \mathfrak a_0=\mathfrak z$.

On the other hand $u.h$ is the sum of $h^\prime$ and $-x$, which commute, $h^\prime$ being ad-semisimple and $-x$ being ad-nilpotent. Thus, since $u.h$ is also ad-semisimple, $x$ must be zero. Finally $u.\eta=\eta$, since $u\in \bf U \subset \Stab_{\bf{Q}_\Lambda}\eta$.

\end {proof}

\begin{rem}
\

 Thus for a truncated biparabolic subalgebra $\mathfrak q_\Lambda$ two adapted pairs $(h,\,\eta), (h',\,\eta')$ (with $h,\,h'$ ad-semisimple) are equivalent in the sense of \ref{defequivAP}, if and only if there exists $q \in \textbf{Q}_\Lambda$ such that $q.h=h',\, q.\eta =\eta'$.
Observe that this result extends to regular truncated biparabolic subalgebras the result in \cite[Prop. 8.10]{J0.7} proved for a regular truncated parabolic subalgebra.
\end{rem}

\subsubsection{}\label{7.9}

Retain the notations and hypotheses of \ref {7.1}.   In the following we give a criterion for deciding when two adapted pairs $(h,\,\eta)$ and $(h^\prime,\,\eta^\prime)$ for $\mathfrak q_{\Lambda}$ are equivalent.

 Up to conjugation by elements of $\textbf{Q}_\Lambda$ we can assume that $h,\,h^\prime \in \mathfrak h_{\Lambda}$. Then by Proposition \ref {equivpara} equivalence implies that there exists $q''\in{\bf Q}_{\Lambda}$ such that $q''.h^\prime = h$. Then as noted in \cite [Lemma 8.11]{J0.7} we can assume that $q''$ belongs to the Weyl group $W$ and hence to the Weyl subgroup $W_M$ corresponding to the Levi factor of $\mathfrak q_{\Lambda}$.  Summarizing

\begin {cor} Let $(h,\,\eta)$ and $(h^\prime,\,\eta^\prime)$ be adapted pairs for $\mathfrak q_{\Lambda}$. Then there exists $q,\,q^\prime \in{\bf Q}_{\Lambda}$ such that $q.h,\,q^\prime. h^\prime  \in \mathfrak h_{\Lambda}$ and if these pairs are equivalent
then there exists $w \in W_M$ such that $wq^\prime. h^\prime =q.h$.
\end {cor}

\subsubsection{}\label{7.9.1}

Continue to retain the notations and hypotheses of \ref{7.1}.
Let $(h,\,\eta)$ be an adapted pair for a truncated biparabolic $\mathfrak q_\Lambda$.  Recall that we may assume $h \in \mathfrak h_{\Lambda}$ and that $\eta$ may be presented in the form given in \ref {7.1}$(*)$, that is through a subset  $S$ of the roots of $\mathfrak q^*$.  Recalling the second remark of  \ref{2.2.2}, we say that such a presentation is minimal if $S_{\mid_{\mathfrak h_\Lambda}}$ is a \textit{basis} for $\mathfrak h^*_\Lambda$.  It is usually easier to construct minimal presentations and indeed all presentations of adapted pairs constructed in \cite {J1} were minimal.

A slight advance in showing that every equivalence class of an adapted pair has a minimal representative has been obtained in \cite {J7}.  Let $(\mathfrak q_\Lambda)_0$ denote the zero eigenspace of $\mathfrak q_\Lambda$ with respect to $\ad h$ and $(\textbf{Q}_\Lambda)_0$ the closed connected subgroup of $\textbf{Q}_\Lambda$ with Lie algebra $\ad_{\q_{\Lambda}}((\mathfrak q_\Lambda)_0)$.  Let $(\mathfrak q^*_\Lambda)_{-1}$ denote the $-1$ eigenspace of $\mathfrak q^*_\Lambda$ with respect to $\ad h$.  By definition $\eta \in (\mathfrak q^*_\Lambda)_{-1}$.  Obviously $(\textbf{Q}_\Lambda)_0$ acts on $(\mathfrak q^*_\Lambda)_{-1}$ and fixes $h$.

By \cite [Prop. 4.4(iv)]{J7}, the orbit $(\textbf{Q}_\Lambda)_0.\eta$ is open dense in $(\mathfrak q^*_\Lambda)_{-1}$. It follows that, if $(h,\,\eta')$ is an adapted pair for $\mathfrak q_\Lambda$, then $\eta'=q.\eta$ for some $q \in (\textbf{Q}_\Lambda)_0$.

Moreover to show that the equivalence class of $(h,\eta)$ has a minimal representative, it is enough to construct (with $h$ fixed) a regular element $\eta' \in (\mathfrak q^*_\Lambda)_{-1}$ of minimal presentation, since then $\eta$ is conjugate to $\eta'$ under $(\textbf{Q}_\Lambda)_0$.

Again suppose that $(h',\,\eta')$ is an adapted pair for $\q_{\Lambda}$ with $\eta'$ of minimal presentation and $h' \in \mathfrak h_\Lambda$.  If this pair is equivalent to the adapted pair $(h,\,\eta)$ (which need not be the case !) then by Corollary \ref{7.9}, there exists $w \in W_M$ such that $h=w.h'$.  Moreover since $W_M$ just permutes the roots of $\mathfrak q^*$, it is clear that $(w.h',\,w.\eta')=(h,\,w.\eta')$ is an adapted pair for $\q_{\Lambda}$ whose second term is of minimal presentation.  Then by the above $\eta$ and $w.\eta'$ are conjugate under $(\textbf{Q}_\Lambda)_0$.

If $\dim (\mathfrak q^*_\Lambda)_{-1}= \dim \mathfrak h_\Lambda$ then every second element of an adapted pair for $\q_{\Lambda}$ with $h$ as a first element is of minimal presentation (by Lemma \ref{2.2.1} $(iii)$).
 It turns out \cite [Sect. 4]{J7} that this condition is equivalent to the condition $(\mathfrak q_\Lambda)_0=\mathfrak h_\Lambda + \mathfrak z_\Lambda$ where we recall that $\mathfrak z_\Lambda$ is the centre of $\mathfrak q_\Lambda$.  However this need not hold and indeed $(\mathfrak q_\Lambda)_0$ can even fail to be a solvable Lie algebra \cite [Sect. 4]{J7}.
By the above we can give the following remark.

\begin{rem}
\

The criterion for equivalence for adapted pairs given in Corollary \ref{7.9} is both necessary and sufficient.
\end{rem}



\subsection{The integrality of the eigenvalues.}\label{7.10}

Let $\mathfrak a$ be an algebraic Lie algebra and continue to assume that $Sy(\mathfrak a)=Y(\mathfrak a)$ and is polynomial.  Given an adapted pair $(h,\,\eta)$ for $\mathfrak a$ one can ask\footnote{Question also posed by Elashvili.} if the eigenvalues of $\ad h$ on $\mathfrak a$ are all integer.  Indeed believing this to be true is why we chose the scaling $(\ad h)(\eta)=-\eta$ rather than the scaling $(\ad h)(\eta)=-2\eta$ which arises from the description of an $\s\l_2$ triple.  In particular our question has a positive answer if $\mathfrak a$ is semisimple.  This  is expressed by saying that the regular nilpotent orbit is even.

Now suppose that $\mathfrak a$ is a truncated biparabolic.  Then our question would have a positive answer if we had the following strengthening of the conclusion of Proposition \ref {7.4}, namely that $\Delta \subset \mathbb ZS + \mathbb ZT$.  However since $S,\,T \subset \Delta$, this is automatic in type $A$ (see for instance \cite[A.2]{J0.5}).  Thus we have the following

\begin {cor}  Let $\mathfrak q_\Lambda$ be a truncated biparabolic subalgebra of $\mathfrak {sl}_n(\bf k)$ and  $(h,\,\eta)$ an adapted pair for $\mathfrak q_\Lambda$ with $h\in\h_{\Lambda}$.  Then the eigenvalues of $\ad h$ on $\mathfrak q_\Lambda$ are integer.
\end {cor}

\begin{rem}
\
  The above conclusion may fail if $S(\a)$ admits non-trivial semi-invariants, as the following example shows.

  \

  Extend the Heisenberg Lie algebra with the one non-trivial relation $[x,\,y]=z$ by an $\ad$-semisimple element $h$, satisfying $[h,\,x]=c\,x, [h,\,y]=(1-c)\,y\, : \, c \in \mathbb Q \setminus \{0,\,1\}$.  Then the resulting algebra $\mathfrak a$ is algebraic and Frobenius. With respect to the above basis, let $\zeta$ be the element dual to $z$.  Then $(h,\,\zeta)$ is an adapted pair for $\mathfrak a$, yet the eigenvalues of $\ad h$ are $\{c,\,(1-c),\,1,\,0\}$ which run over all rational values.  Though $Sy(\mathfrak a)$ and $Y(\mathfrak a)$ are polynomial they are not equal.

\end{rem}

\section{ Weierstrass sections.}\label{7.11}

Recall the definition of a Weierstrass section given in \ref{1.2}.

\

Before addressing the main goal of this section we take this opportunity of correcting some bad points in the presentation given in \cite [7.6-7.8]{J2}. In conformity with the notation of \cite [Sect. 7]{J2} we adopt the more general context of \cite [Sect. 7]{J2} of a connected algebraic group $\textbf{A}$ acting linearly and regularly on a finite dimensional vector space $X$.  In this $y+V$ rather than $\eta+V$ is used to denote a linear subvariety of $X$.  Let $\mathfrak a$ denote the Lie algebra of $\textbf{A}$.

 \subsection {Corrections}\label {7.11.1}

 \

 1)  In the statement of \cite[Prop. 7.8(i)]{J2} the hypothesis that $V$ is a complement to $\mathfrak a.y$ in $X$ was inadvertently forgotten.  It is clear that it was intended from the first line of the proof.

 \

 2) The comment ``equivalently $d=1$ in the above" occurring in the second paragraph of \cite[7.6]{J2} is incorrect (it was not used there). A counter-example was already provided in that paper \cite [11.4, Example 3]{J2}!

 \

 3) It was claimed in \cite [7.7]{J2} that $\mathscr S=\{s\in y+V\mid\mathfrak a.s\cap V=0\}$ is open in $y+V$. This is probably false.  Let us show however that it is a finite union of irreducible locally closed sets.  This serves the same purpose particularly concerning \cite [7.7]{J2}. Notably $\mathscr S_{reg}$ is open in $y+V$.

 \

 Set $T=y+V$.  Given $d \in \mathbb N$, set $T_d:=\{t \in T\mid\dim \mathfrak a.t=d\}$.  A standard argument (involving the complement to the zero set of appropriate minors - see \cite[Prop. 1.11.5]{Dix}), shows that $T_d$ is open in $T$ if $d$ takes its largest value.  Then for arbitrary $d$, induction shows that $T_d$ is locally closed in $T$.  Thus $T$ is a finite union of irreducible locally closed sets.  Let $U$ be such a subset of $T$, say $U \subset T_d$. Then $U':= \{u \in U\mid\mathfrak a.u\cap V=0\}=\{u \in U\mid\dim (\mathfrak a.u +V) = d +\dim V\}$.  The latter can be expressed as the complement to the zero set of appropriate minors and so is open in $U$. It is either empty or irreducible.  Hence the assertion.

 \

 4) Assume that $\textbf{A}.(y+V)$ is dense in $X$.  It was claimed in \cite [7.7]{J2} that $O:=\{s \in y+V\mid\mathfrak a.s+V=X\}$ is non-empty.  This is true and we explain why.

 \

  Let $T_{u,\,U}$ denote the tangent space at some $u \in U \subset X$.  Then for $s\in y+V$, $\mathfrak a.s=T_{s,\,\textbf{A}.s}, T_{s,\,y+V}=V,\, T_{s,\,X}=X$.  Consider the morphism $\varphi:\textbf{A}\times (y+V) \rightarrow X$ defined by the group action. By the hypothesis and \cite[Thm. 16.5.7 (ii)]{TY}, the tangent map $d\varphi_{(a,\,s)}\,:\,T_{(a,\,s),\,\textbf{A}\times(y+V)} \rightarrow T_{\varphi(a,\,s),\,X}$ is surjective at some point $(a,s) \in \textbf{A}\times (y+V)$.   Its image is $\mathfrak a.(a.s) +a.V=a.[(Ad a^{-1}(\mathfrak a)).s+V]=a.[\mathfrak a.s+V]$, which since $a \in \bf A$ is invertible, has the same dimension as $\mathfrak a.s+V$.  Thus $s \in O$, which is hence non-empty.

 \

 5) Assume that the algebra $R[X]$ of regular functions on $X$ has no proper semi-invariants (as an $\textbf{A}$ module).  Assume that $(y+V)_{reg}:=(y+V)\cap X_{reg}$ is non-empty, hence open dense in $y+V$.  Assume that the restriction map $R[X]^{\textbf{A}} \rightarrow R[y+V]$ induces an isomorphism of rings of fractions.  Then as explained in \cite [7.5]{J2} it follows that $\dim V$ is the codimension of an $\bf A$ orbit in $X_{reg}$, that is $\ell_X(\a)$ with notations of \ref{1.2}. Let $d'$ be the common denominator of the $\dim V$ generators of the algebra $R[y+V]$ when expressed as fractions in the elements of $R[X]^\textbf{A}$ and $D'$ the zero set in $y+V$ of $d'$. By construction $d'$ is the restriction of an element $d \in R[X]^\textbf{A}$. Let $D$ be the zero set of $d$ in $X$.   Obviously $D'=(y+V)\cap D$ and $(y+V)_{reg} \setminus (D'\cap(y+V)_{reg})$ is non-empty (since otherwise $D'$ would be equal to $y+V$
 and then $d'$ would be zero, which is impossible by construction).  Thus $R[X]^\textbf{A}$ separates the points of $(y+V) \setminus D'$, whilst $d$ further separates them from those of $D'$.

 \

 $(*)$.  In particular an $\textbf{A}$ orbit through a point of $(y+V)\setminus D'$ cannot pass through a different point of $y+V$.

 \

 \

 It was claimed in \cite [7.7]{J2} that as a consequence of the above, $\textbf{A}.(y+V)$ is dense in $X$.  This is true and we explain why.

 \

  Since $\textbf{A}\subset GL(X)$ is assumed connected and $V$ is a vector subspace of $X$, both $\textbf{A}$ and  $y+V$ are closed and irreducible. Thus ${\bf A}\times (y+V)$ is an irreducible subvariety of $GL(X)\times X$ (\cite[Thm. 3 of Section 3 in Chap. I]{Scha}).

  Let $\varphi\,:\,{\bf A}\times (y+V)\rightarrow X$ be the morphism defined by group action. Its image
${\bf A}.(y+V)$ is again irreducible and contains an open subset of its closure.

 By \cite[Thm. 7 of Section 6 in Chap. I]{Scha}  we obtain
 $$\dim {\bf A}.(y+V)\ge \dim({\bf A}\times (y+V))-\dim\varphi^{-1}(a.s), \forall a \in \textbf{A}, s \in y+V .\eqno {(**)}$$

  Take more particularly $s\in (y+V)_{reg}\setminus (D'\cap(y+V)_{reg})$.  By $(*)$ above, one has  for $a\in{\bf A}$, $\varphi^{-1}(a.s)=(a{Stab}_{\bf A}s,s)$. Moreover $s \in X_{reg}$, so $\dim{Stab}_{\bf A}s=\dim \textbf{A} +\dim V - \dim X $. Thus by $(**)$, $\dim{\bf A}.(y+V)=\dim X$, and so ${\bf A}.(y+V)$ is dense in $X$.

 \subsection {} \label {7.11.2}

From now on we revert to the hypotheses on $\a$ and $\bf A$ described at the beginning of this paper (see \ref{1}) 
and to $\eta+V$ denoting a linear subvariety of $X=\mathfrak a^*$.

\begin {lemma}    Assume that $Sy(\mathfrak a)=Y(\mathfrak a)$.  Given

\

 (i) The restriction map $\varphi$ induces an algebra isomorphism of $Y(\mathfrak a)$ onto the algebra $R[\eta+V]$ of regular functions on $\eta+V$.

\

Then

\

(ii)  ${\bf A}.(\eta+V)$ is dense in $\mathfrak a^*$.  In particular $(\eta+V)_{reg}\neq \emptyset$.

\

(iii) $\eta+V$ meets every co-adjoint orbit at most once and then transversally.

\end {lemma}


\begin {proof}  As explained in \cite [7.9]{J2}, the first part of $(ii)$ is a consequence of injectivity in $(i)$ and a theorem of Dixmier, Duflo and Vergne.  It implies that $(\textbf{A}.(\eta+V))_{reg} \neq \emptyset$ (since otherwise, by the first part of $(ii)$, we would have $\a^*_{reg}=\emptyset$ which is never the case), hence the second part of $(ii)$. The first part of $(iii)$ is a consequence of surjectivity.  It remains to prove the last part of $(iii)$.  For this (see \ref{1.5} and recalling that $T_{\xi,\,\eta+V}=V$) we must prove that $\mathfrak a.\xi\cap V=\{0\}$, for all $\xi \in \eta+V$.

Take $ \xi \in \mathfrak a^*$.  As noted in \cite [5.4$(**)$]{JS} for example, one has $df(\xi) \in (\mathfrak a.\xi)^\perp$ for all $f\in Y(\a)$.  Thus
$$df(\xi)(v)=0, \forall f \in Y(\a), v \in \a.\xi\cap V. \eqno {(*)}$$

Fix $\xi \in \eta+V$.


%
%
%

Suppose $\eta \in V$.  Then $R[\eta+V]=S(V^*)$ and the isomorphism in $(i)$ implies that $Y(\mathfrak a)=S(\mathfrak z)$, where $\mathfrak z$ is the centre of $\mathfrak a$.  Take $v \in \mathfrak a.\xi\cap V$.  Then for all $z \in \mathfrak z$ we have $z(v)=0$, which by surjectivity of $\varphi$ forces $v=0$.

It remains to consider the case $\eta \notin V$.

Set $\ell = \dim V, n= \dim \mathfrak a$. Let $(x_i^*)_{i=1}^{\ell}$ be a basis of $V$, complete $(x_i^*)_{i=1}^{\ell}\cup\{\eta\}$ to a basis of $\a^*$ and take its dual basis $(x_i)_{i=1}^n$ in $\a$.
Then $\{x_i\}_{i=1}^\ell$ is a basis of $V^*$ and $y:=x_{\ell+1}$ vanishes on $V$ equaling $1$ on $\eta$.


In particular $y(\xi)=y(\eta)=1$ and so for any polynomial $p \in \textbf{k}[y]$ one has $p(\xi)=p(\eta)$.

Let $\Phi:Y(\mathfrak a)\rightarrow R[\textbf{k}\eta\oplus V]$ be the restriction map. Given $f \in Y(\mathfrak a)$, we have $\Phi(f) \in \textbf{k}[x_1,x_2,\ldots,x_\ell,y]$.  The surjectivity of $\varphi$ means
that for all $i =1,2,\ldots,\ell$, there exists $f_i \in Y(\mathfrak a)$ such that for all $v\in V$, $\varphi(f_i)(\eta+v)=x_i(v)=x_i(\eta+v)$.

Given $I \in \mathbb N^\ell$, let  $x^I$ denote the corresponding monomial in $\{x_i\}_{i=1}^\ell$. Then there exist a finite set $\mathscr F \subset \mathbb N^\ell$ and polynomials $p,q_I:I\in \mathscr F$ of $y$ satisfying $p(\eta)=1,q_I(\eta)=0:I \in \mathscr F$ such that $\Phi(f_i)=px_i+\sum_{I \in \mathscr F} q_Ix^I$.

By choice of the basis of $\mathfrak a$ we have for all $v\in V$, $x_j(v)=0$ for $j>\ell$ and $x_j(\xi)=0$ for all $j>\ell+1$. Thus for all $j\leq \ell$ one has $(\partial f_i/\partial x_j)(\xi)=(\partial \Phi(f_i)/\partial x_j)(\xi)$. Then through the above expression for $\Phi(f_i)$ and the choice of $p,q_I:I \in \mathscr F$, we obtain
$$df_i(\xi)(v)=\sum_{j=1}^\ell x_j(v)(\partial f_i/\partial x_j)(\xi)=x_i(v), \forall  v \in V.$$

Thus for $v \in V$, the condition $df_i(\xi)(v)=0$, for all $i=1,2,\ldots,\ell$ implies $v=0$.

 It remains to take $v \in \mathfrak a.\xi\cap V$ and to apply $(*)$,

\end {proof}

 \subsection {} \label {7.11.3}

 In the language of \ref{1.5}, the above lemma means that if ${S}(\mathfrak a)$ has no proper semi-invariants, then a Weierstrass section for $\a$ is an affine slice for $\a$.

 Notice that $(\eta +V)_{reg}$ being non-empty, implies that it is open dense in $\eta+V$. Thus $D'':=(\eta +V) \setminus (\eta+ V)_{reg}$ is closed of codimension $\geq 1$ (for $V$ not reduced to zero). We remark that Example 3 of \cite [11.4]{J2} shows that the hypothesis of Lemma \ref {7.11.2} does not exclude $D''$ being of codimension $1$.  In this example one may verify that $(iii)$ holds by explicit computation.

 If we assume that $D''$ has codimension in $\eta+V$ at least $2$, then the converse to Lemma \ref {7.11.2} holds. Indeed conditions $(ii)$, $(iii)$ of Lemma \ref {7.11.2} imply that the locally closed subset $\mathscr S=(\eta+V)_{reg}$ is a slice to the action of $\bf A$ on $\a^*$ which is affine in the sense of \ref{1.5}.  The additional hypotheses of ${S}(\mathfrak a)$ having no proper semi-invariants and $D''$ being of codimension $\geq 2$, complete the hypotheses of \cite [Prop. 7.4]{J2} from which $(i)$ of Lemma \ref {7.11.2} results.

 \subsection {} \label {7.11.4}

It seems unlikely that the converse to Lemma \ref {7.11.2} holds in general. Let us examine however what seems to be the most natural line of attack.


Set $T=\eta+V, X=\mathfrak a^*$.  Let $m_1=\max_{t \in T}\dim \mathfrak a.t$ and set $O_1:=\{t \in T|\dim \mathfrak a.t=m_1\}$, which is open in $T$ and non-empty.  We repeat the construction with $C_2:=T\setminus O_1$. This expresses $T$ as a finite disjoint union of  locally closed sets $O_i,\,1\le i\le r$,
with $C_i:=\sqcup_{j=i}^rO_j$ ($C_1=T$) the closure of $O_i$ in $T$ and such that the dimension of an $\textbf{A}$ orbit through $O_i$ takes some constant value $m_i$.   Set $n_i:=\dim O_i$.


 For each $i \in I=\{1,\,\ldots,\,r\}$, let $E_i$ be the union of the irreducible components of $C_i\setminus O_i$ of codimension $1$ in $C_i$.  By Krull's theorem the ideal of definition $I(E_i)$ of the closed set $E_i$ is principal, that is $I(E_i)=d_iR[C_i]$, for some $d_i \in R[C_i]$.

 Let $A$ be a subalgebra of $R[T]$, that is to say we have an injective homomorphism $\varphi:A \rightarrow R[T]$. For all $i \in I$ let $p_i$ be the canonical projection of $R[T]$ onto $R[C_i]$ and set $\varphi_i=p_i\varphi$.

 \begin {lemma} Suppose that $R[T]/\varphi(A)$ is finitely generated as an $A$ module.  Further for all $i \in I$ suppose that

 \

 (i) $d_i \in \varphi_i(A)$, that is $d_i=\varphi_i(d_i'')$ for $d_i''\in A$.

 \

 (ii) for all preimage $d_i''$ of $d_i$, $\varphi_i$ induces a surjection of $A[d_i''^{-1}]$ onto $R[C_i][d_i^{-1}]$.
%
%

 \

 Then $\varphi$ is surjective.
 \end {lemma}

 \begin {proof}  It is clear that the hypothesis of finite generation passes to each quotient $R[C_i]/\varphi_i(A)$.  Then through induction on $\mid\! I\!\mid$ it suffices to consider the case $r=\mid\! I\!\mid=2$.  If $d_1=1$, there is nothing to prove.  Otherwise we write $R[C_1]=R, d=d_1$. Note that $d_2=1$ (and we can take $d''_2=1$ too) since $O_r=C_r$.

 Identify $A$ with its image in $R$.  By $(i)$, $d \in A$.  Consider $M:=R/A$ as an $A$ module. By $(ii)$ with $i=1$, we obtain an isomorphism $A[d^{-1}] \iso R[d^{-1}]$, thus $M$ has $d$-torsion, so being finitely generated is annihilated by a power of $d$.  By $(ii)$ with $i=2$, taking $d_2''=1$, we obtain a surjection $A\twoheadrightarrow R[C_2] \twoheadrightarrow R/Rd$, that is  $R=A+dR$ and so $dM=(dR+A)/A=R/A=M$, forcing $M=0$, as required.

 \end {proof}

\subsection {} \label {7.11.5}

We now show how the above lemma gives under two further conditions, a converse to Lemma \ref {7.11.2}.  Let $\eta +V$ be a linear subvariety of $\mathfrak a^*$ in \ref {7.11.4} and retain the notation there. In addition let $S_i$ denote the closure of $\textbf{A}.O_i$.  Set $A=Y(\mathfrak a)$ and let $\varphi: A \rightarrow R[\eta+V]$ be defined by restriction of functions.

\begin {lemma}  Assume that $Sy(\mathfrak a)=Y(\mathfrak a)$.  Given

\

(i)  ${\bf A}.(\eta+V)$ is dense in $\mathfrak a^*$.

\

(ii) $\eta+V$ meets every co-adjoint orbit at most once and then transversally.

\

(iii) $S_i$ is smooth for all $i \in I$.

\

(iv) $R[\eta+V]/\varphi(Y(\mathfrak a))$ is finitely generated as an $Y(\mathfrak a)$ module.

\

(v) For all $i\in I$, the algebra of $\bf A$ semi-invariants (resp. invariants) in $R[S_i]$ is an image by restriction of functions of $Sy(\a)$ (resp. $Y(\a)$).

\

Then

\

(vi) $\varphi$ is an isomorphism of $Y(\mathfrak a)$ onto $R[\eta+V]$.

\end {lemma}

\begin {proof}  By $(i)$, $\varphi$ is an embedding.

 For all $i \in I$, let $O_i^j:j \in J_i$ denote the set of irreducible components of $O_i$.

 Consider the morphism $\psi:\textbf{A}\times O_i^j \rightarrow \textbf{A}.O_i^j$ defined by the action. By the first
 part of $(ii)$ the fibre over the point $a.s: a \in \textbf{A}, s \in O_i^j$ is just
$(a\Stab_{\textbf{A}}s,\,s)$.  By
construction this has constant dimension and so from \cite [Thm. 7 of Sect. 6 in Chap. I]{Scha} it follows that $\dim S_i=\dim \textbf{A}.O_i=m_i+n_i$ in the notation of \ref {7.11.4}. On the other hand the image of the tangent map $d\psi_{(a,\,s)}:T_{(a,\,s),\,\textbf{A} \times O_i } \rightarrow T_{a.s,\,S_i}$ has dimension independent of $a \in \textbf{A}$.  Moreover by transversality the image at the identity in $\bf A$ is the direct sum of the tangent space to the orbit through $s \in O_i$, which has dimension $m_i$ and the tangent space $T_{s,\,O_i}$ which has dimension at least $\dim O_i=n_i$.  Thus the dimension of the image of $d\psi_{(a,\,s)}$ is at least $\dim S_i$, for all $(a,\,s) \in \textbf{A}\times O_i$.

 We conclude from $(iii)$ that $d\psi_{(a,\,s)}$ is surjective for all $(a,\,s) \in \textbf{A} \times O_i$.  Thus $\psi$ is a smooth map whose image $\textbf{A}.O_i$ is open in $S_i$.  Then by a result of Hinich \cite [Prop. 12.3]{J2} restriction gives an isomorphism $\varphi'_i:R[\textbf{A}.O_i]^\textbf{A}\iso R[O_i]$.

 Let $D_i$ be the union of the irreducible components of codimension one in $S_i$ of the $\textbf{A}$ invariant closed set $S_i \setminus \textbf{A}.O_i$.  Its ideal of definition is principal with generator $d_i\in R[S_i]$ which must be a semi-invariant.
 
Now $S_i$ is a closed $\textbf{A}$ invariant subvariety of $\mathfrak a^*$, and by construction, $\a^*=S_1=\overline{{\bf A}.C_1}$
 by $(i)$ and  for all $2\le i\le r$, $S_i=\overline{{\bf A}.C_i}=S_{i-1}\setminus{\bf A}. O_{i-1}=\sqcup_{j=i}^r{\bf A}.O_j$. Moreover by $(v)$ and the hypothesis that $Sy(\a)=Y(\a)$ we deduce that  $d_i\in R[S_i]^{\bf A}$.
 


Then $R[\textbf{A}.O_i]=R[S_i][d_i^{-1}]$ and so $R[\textbf{A}.O_i]^\textbf{A}=R[S_i]^\textbf{A}[d_i^{-1}]$.  Similarly we can write $R[O_i]=R[C_i][d_i^{\prime-1}]$, for some $d_i' \in R[C_i]$.  Thus $\varphi'_i$ induces an isomorphism of $R[S_i]^\textbf{A}[d_i^{-1}]$ onto $R[C_i][d_i^{\prime-1}]$.   Since in addition $\varphi'_i$ induces an embedding to $R[S_i]^\textbf{A}$ into $R[C_i]$, it follows that $\varphi_i'(d_i)=d_i'u_i$ for some unit $u_i \in R[C_i]$.  

Yet $R[C_i][{d'_i}^{-1}]=R[C_i][({d'}_iu_i)^{-1}]$ so we may assume $u_i=1$ without loss of generality.

 Identifying $d_i\in R[S_i]$ with its image $\varphi'_i(d_i)$ in $R[C_i]$ and setting $d_i={d_i''}_{\mid C_i}=p_i\varphi(d_i'')=\varphi_i(d_i'')$ with $d_i''\in Y(\a)$ (via the surjection in $(v)$), we conclude that $\varphi_i$ induces a surjection of $Y(\a)[d_i''^{-1}]$ onto $R[C_i][d_i^{-1}]$.  Thus conditions $(i)$ and $(ii)$ of Lemma \ref {7.11.4} are satisfied.   Finally $(iv)$ implies the initial hypothesis of Lemma \ref {7.11.4}.  Then $(vi)$ of the present lemma follows from the conclusion of Lemma \ref {7.11.4}.

\end {proof}

\subsection {Comments} \label {7.11.6}

Recall notations of \ref{7.11.3}.
It might seem surprising that the general case does not seem to go through without significant extra conditions, whilst the special case when $D''$ has codimension $\geq 2$ is rather easy.  On the other hand codimension $2$ conditions are somewhat ubiquitous in invariant theory.

It may be that the condition that $S_i$ be smooth can be weakened to just $\textbf{A}.O_i$ being smooth.  By condition $(ii)$ of Lemma \ref {7.11.5}, $\textbf{A}.O_i$ is a fibration with fibres being $\textbf{A}$ orbits of constant dimension with base $O_i$.  Yet this may not be too helpful as it is not so obvious if $O_i$ is smooth.  Indeed the latter is given by the non-vanishing of certain minors and vanishing of others, so is unlikely to be smooth.

The conclusion in the proof of Lemma \ref {7.11.5} that the image of $d_i$ and $d_i'$ are proportional implies (taking just $i=1$) that $D''$ has codimension $\geq 2$ if and only if $\textbf{A}.(\eta+V)\setminus{\bf A}.(\eta+V)_{reg}$ has codimension $\geq 2$.

The only examples we have when $D''$ does not have codimension $\geq 2$ are when $\mathfrak a^*_s:=\mathfrak a^* \setminus \mathfrak a^*_{reg}$ has codimension $1$.  We had thought that this may be a general fact and indeed obtain following a comment in Popov \cite [Thm. 2.2.15 b)]{P}.  However here Popov is assuming that the restriction map $\varphi$ is smooth at $\eta$.  In any case if one assumes that $\textbf{A}.(\eta+V) \subset \mathfrak a^*_{reg}$, then such a conclusion is immediate.  However the latter condition is only satisfied rather rarely.

 \subsection {} \label {7.11.7}

  Unfortunately we do not have an example of an affine slice $\eta+V$ for $\a$ with $D''=\eta+V\setminus (\eta+V)_{reg}$ of codimension $1$ which is not a Weierstrass section to justify the extra conditions imposed in Lemma \ref {7.11.5}.


 The following example is of a slice $\mathscr S$ to the action of $\bf{A}$ on $\a^*$ which is of codimension $0$ in a linear subvariety $\eta+V$. It is noted that $\bf{A}$ orbits may pass more than once through $D:=\eta+V\setminus \mathscr S$ and moreover need not pass transversally. Thus $\eta+V$ is \textit{not} an affine slice for $\a$ (whilst the locally closed subset $\mathscr S$ is an affine slice).   Thus by Lemma \ref {7.11.2}, $\eta+V$ cannot be a Weierstrass section and so we cannot deduce that $Y(\mathfrak a)$ is polynomial. In fact  Dixmier \cite {Dix0.1} verified by explicit computation that $Y(\a)$ is not polynomial.

 Recall example 4 of \cite [11.4]{J2}.   In this $\mathfrak a$ is the standard filiform of dimension $5$.  It has basis $\{x_i\}_{i=1}^5$ with relations $[x_1,x_i]=x_{i+1}\,:\,i=2,\,3,\,4$, all other brackets being zero. $\a$ is nilpotent and $Y(\mathfrak a)$ is polynomial after localisation at $x_5$.  Let $\{y_i\}_{i=1}^5$ be a dual basis, set $\eta=0$ and $V$ the span of $y_2,\,y_3,\,y_5$.  Let $D$ be the zero locus of the invariant $x_5$ in $V$. One has $D=\mathbb Cy_2+\mathbb Cy_3$. Set $\mathscr S = V\setminus D$.

 As noted in loc cit, $\mathscr S = \mathscr S_{reg}$ and the $\textbf{A}$ orbits through $V$ are transversal (exactly) on $\mathscr S$, those through $\mathscr S$ are separated by $Y(\mathfrak a)$ and $\mathfrak a^*\setminus \textbf{A}.\mathscr S_{reg}$ is the zero locus of the invariant $x_5$.

 Let us examine the orbits through $D$.

One has $(\ad x_1)(y_3)=-y_{2}$ and $(\ad x_2)(y_3)=y_1$. The coadjoint action of the other generators of $\a$ on $y_3$ is zero and for all $x_i\in\a$,  $(ad\,x_i)(y_2)=0$.

Thus $$e^{d \ad x_2}e^{c \ad x_1}(ay_3+by_2)=ay_3+(b-ac)y_2+ady_1.$$

This vector lies in $D$ if and only if $ad=0$ that is $a=0$ or $d=0$.
Moreover equating this with a corresponding expression for primed quantities we obtain
$$a=a', b-ac=b'-a'c'.$$

Thus $$a=a', \quad b-b'=a(c-c').$$

 If $a=0$, then $b=b'$, so $D$ meets every ${\bf A}$ orbit of an element in $\mathbb Cy_2$ exactly once.  Otherwise $b-b'$ can be arbitrary.  Thus $ay_3+\mathbb Cy_2: a\neq 0$ lies on a single ${\bf A}$ orbit and the intersection of $D$ with the ${\bf A}$ orbit of an element $ay_3+by_2$, ($a\neq 0$) is $ay_3+\mathbb Cy_2$.

 In conclusion, $D$ admits a one parameter family of orbits which pass just once through $D$ and a one-parameter family of orbits each of which meets $D$ in a line.

Again $(d\ad x_2 +c \ad x_1)(ay_3+by_2)=ady_1-acy_2$, which lies in $D$ if and only if $ad=0$.

 Either $a=0$ or $d=0$.  In the first case the orbits cut $D$ transversally (at exactly one point).  In the second case transversality fails exactly when $c\neq 0$. Indeed $(c\ad x_1)(ay_3) =-cay_2\in D$, which is non-zero if $c\neq 0$.

Thus both parts of \ref {7.11.2}$(iii)$ fail.

 Finally one may check that the fundamental (semi)-invariant of $\mathfrak a$ is scalar and so $\mathfrak a^* \setminus \mathfrak a^*_{reg}$ has codimension $\geq 2$ (by \cite[4.1]{JS}).  Consequently $\mathscr S$ does not meet all the regular orbits in $\mathfrak a^*$.


\begin{thebibliography}{m}


\bibitem {C}   C. Chevalley,  Invariants of finite groups generated by reflections. {\it Amer. J. Math. {\bf 77} (1955), 778 -- 782}.

\bibitem {Dix0} J. Dixmier, Sur les rep\'esentations unitaires des groupes de Lie nilpotents. II. (French) {\it Bull. Soc. Math. France {\bf 85} (1957) 325 -- 388}.

\bibitem {Dix0.1}  J. Dixmier,  Sur les repr\'esentations unitaires des groupes de Lie nilpotents. III. {\it Canad. J. Math. {\bf 10} (1958) 321-- 348}.

\bibitem {Dix} J. Dixmier, Alg\`ebres enveloppantes, Editions Jacques Gabay, Gauthier-Villars, 1974.



\bibitem{F} F. Fauquant-Millet, Sur la polynomialit\'e de certaines alg\`ebres d'invariants d'alg\`ebres de Lie. {\it M\'emoire d'Habilitation \`a Diriger des Recherches, http://tel.archives-ouvertes.fr/tel-00994655}.

\bibitem {FJ0} F. Fauquant-Millet and A. Joseph, Semi-centre de l'alg\`ebre enveloppante d'une sous-alg\`ebre parabolique d'une alg\`ebre de Lie semi-simple. (French) [Semicenter of the enveloping algebra of a parabolic subalgebra of a semisimple Lie algebra]  {\it Ann. Sci. \'Ecole Norm. Sup. (4)  {\bf 38}  (2005),  no. 2, 155 -- 191}.


\bibitem{FJ2} F. Fauquant-Millet and A. Joseph, La somme des faux degr\'es - un myst\`ere en th\'eorie des invariants (French) [The sum of the false degrees - A mystery in the theory of invariants] {\it Adv. in Math.  {\bf Vol 217/4} (2008), 1476 -- 1520}.
\bibitem{FJ3} F. Fauquant-Millet and A. Joseph, {\it in preparation}.


\bibitem {J0.01} A. Joseph,  On a Harish-Chandra homomorphism. {\it C. R. Acad. Sci. Paris, t. {\bf 324}, S\'erie I, (1997), no. 7, 759 -- 764}.


\bibitem {J0} A. Joseph, On semi-invariants and index for biparabolic (seaweed) algebras. I.  {\it J. Algebra {\bf  305}  (2006),  no. 1, 487 -- 515}.

\bibitem {J0.1} A. Joseph, On semi-invariants and index for biparabolic (seaweed) algebras. II. {\it J. Algebra  {\bf 312}  (2007),  no. 1, 158 -- 193}.

\bibitem {J0.5} A. Joseph, A slice theorem for truncated parabolics of index one and the Bezout equation.  {\it Bull. Sci. Math.  {\bf 131}  (2007),  no. 3,  276 -- 290}.

\bibitem {J0.7} A. Joseph, Parabolic actions in type A  and their eigenslices.  {\it Transform. Groups {\bf 12}  (2007),  no. 3, 515 -- 547}.

\bibitem {J1} A. Joseph, Slices for biparabolic coadjoint actions in type A.  {\it J. Algebra  {\bf 319}  (2008),  no. 12, 5060 -- 5100}.

\bibitem {J2} A. Joseph, An algebraic slice in the coadjoint space of the Borel and the Coxeter element. {\it Adv. in Math. {\bf 227} (2011), 522 -- 585}.


\bibitem {J4} A. Joseph,  Compatible adapted pairs and a common slice theorem for some centralizers. {\it Transform. Groups {\bf 13} (2008), no. 3-4, 637-- 669}.

\bibitem {J7}  A. Joseph, The Integrality of the Spectrum of an Adapted Pair. {\it Transform. Groups, to appear}.


\bibitem {JS} A. Joseph and D. Shafrir, Polynomiality of invariants, unimodularity and adapted pairs.  {\it Transform. Groups  {\bf 15}  (2010),  no. 4, 851 -- 882}.

\bibitem{Ko0} B. Kostant, The principal three-dimensional subgroup and the Betti numbers of a complex simple Lie group. {\it Amer. J. of Math.
{\bf 81} (1959),  973 -- 1032.}

\bibitem{Ko} B. Kostant, Lie group representations on polynomial rings. {\it Amer. J. of Math. {\bf 85} (1963),  327 -- 404}.




\bibitem {OV} A. I. Ooms and M. Van den Bergh,  A degree inequality for Lie algebras with a regular Poisson semi-center. {\it J. Algebra {\bf  323}  (2010),  no. 2, 305 -- 322}.

\bibitem{PPY} D. Panyushev, A. Premet, O. Yakimova, On symmetric invariants of centralizers in reductive Lie algebras. {\it J. Algebra {\bf 313} (2007), no. 1, 343 -- 391.}

\bibitem {P}   V. L. Popov, Sections in invariant theory.{\it The Sophus Lie Memorial Conference (Oslo, 1992), 315 -- 361, Scand. Univ. Press, Oslo, 1994}.


\bibitem {RLYG} M. Reeder, P. Levy, J-K. Yu, and B. H. Gross,  Gradings of positive rank on simple Lie algebras. {\it Transform. Groups {\bf 17} (2012), no. 4, 1123 -- 1190}.


\bibitem{RV} R. Rentschler and M. Vergne, Sur le semi-centre du corps enveloppant d'une alg\`ebre de Lie. (French). {\it Ann. Sci. \'Ecole Norm. Sup. (4) {\bf 6} (1973), 389 -- 405.}


\bibitem {R} M. Rosenlicht,  A remark on quotient spaces. {\it An. Acad. Brasil. Ci. {\bf 35} (1963) 487 -- 489}.



\bibitem{Scha} I.R. Schafarevich, Basic Algebraic Geometry, Varieties in Projective Space, 1, Springer-Verlag, Berlin, Heidelberg, New-York, 1977, 1994.

\bibitem{TY} P. Tauvel and R.W.T. Yu, Lie algebras and Algebraic Groups, Springer Monographs in Mathematics, Springer Verlag Berlin-Heidelberg, 2010.

\bibitem {Y} O. Yakimova,  A counterexample to Premet's and Joseph's conjectures. {\it Bull. Lond. Math. Soc. {\bf 39}  (2007),  no. 5, 749 --754}.



\end{thebibliography}
\end{document}